\documentclass[12pt]{article}
\usepackage{amsmath, amsfonts, amsthm, amssymb, verbatim}
\usepackage{graphicx}
\usepackage[mathcal]{euscript}
\usepackage{amstext}
\usepackage{amssymb,mathrsfs}
\usepackage{bbm}
\usepackage[latin1]{inputenc}
\newcommand{\R}{\mathbb R}
 
 \newcommand{\N}{\mathbb N}

  \newcommand{\E}{\mathbb E}

\newcommand{\PP}{\mathbb P}

\newcommand{\calF} {\ensuremath {\mathcal{F}}}

\newcommand{\vp}{\varphi}

\newcommand \loc    {\text{loc}}

\newtheorem{theorem}{Theorem}[section]

 \newtheorem{remark}[theorem]{Remark}
\newtheorem{lemma}[theorem]{Lemma}

\newtheorem{definition}[theorem]{Definition}
\newtheorem{hypothesis}[theorem]{Hypothesis}

\begin{document}

\title{   Regularization by noise in one-dimensional continuity equation. }

\author{ Christian Olivera\footnote{Departamento de Matem\'{a}tica, Universidade Estadual de Campinas, Brazil.
E-mail:  {\sl  colivera@ime.unicamp.br}.
}}

\date{}

\maketitle

\textit{Key words and phrases.
Stochastic partial differential equation, Continuity  equation,  Regularization by noise, It\^o-Wentzell-Kunita formula, Low regularity.}

\vspace{0.3cm} \noindent {\bf MSC2010 subject classification:} 60H15, 
 35R60, 
 35F10, 
 60H30. 


%
\begin{abstract}
A linear stochastic continuity  equation with non-regular coefficients
is considered. We prove existence and uniqueness of strong
solution, in the probabilistic sense,  to the Cauchy problem when the vector field has low regularity,  in which the classical  DiPerna-Lions-Ambrosio theory of uniqueness   of distributional solutions  does not apply. We solve partially the open problem  that is the case 
when the vector-field has  random  dependence. In addition, we prove a stability result for the solutions.

 \end{abstract}
%
\maketitle

%

\section {Introduction} \label{Intro}

Last decades the continuity equation  has attracted a lot
of scientific interest. The reason is that  arises in
a variety of domains such as biology, particle physics, population dynamics, crowd
modeling,  that can be modeled by the
 continuity/ transport equation,
\begin{equation}\label{trasports}
    \partial_t u(t, x) + div ( b(t,x)  u(t,x) ) = 0 \, ,
\end{equation}

\noindent where $u$ is the physical
quantity that evolves in time. Such quantities are for instance the vorticity of a fluid, or the density of a collection of
particles advected by a velocity field which is highly irregular, in the sense that it has a derivative given by
a distribution and a nonlinear dependence on the solution u. 

\noindent  When the coefficients are regular  the unique solution is found  by the method of characteristics.  
Recently research activity has been devoted to study continuity/transport equations with rough coefficients, showing a well-posedness result.
A complete theory of distributional solutions, including existence, uniqueness and
stability properties, is provided in the seminal works of DiPerna and Lions \cite{DL} and Ambrosio  \cite{ambrisio}.

The approach of DiPerna, Lions and Ambrosio relies on the theory of renormalized
solutions. Roughly speaking, renormalized solutions are distributional solutions to
which the chain rule applies in the sense that, for every suitable  $\beta$, $\beta(u)$ 
solves the following continuity equation :

\begin{equation}\label{renord}
    \partial_t \beta (u)+  b  \nabla \beta (u) + div(b) \beta^{\prime}(u) u  = 0 .
\end{equation}

 Whether distributional solutions are renormalized solutions depends on the regularity of $b$. In the paper by DiPerna and Lions was proved that  when $b$ has $W^{1,1}$ spatial regularity (together with a condition of boundedness on the divergence) the commutator lemma between smoothing convolution and  weak solution can be proved and, as a consequence, all $L^{\infty}$-weak solutions are renormalized.
 L. Ambrosio \cite{ambrisio} generalized the theory  to the case of only $BV$ regularity for b instead of $W^{1,1}$. 
Another approach giving explicit compactness estimates has been introduced
in \cite{cripa2}, and further developed in \cite{Bouchut1,jabe},   see also the references therein.
In the case of two-dimensional vector-field, we also refer to the work of F. Bouchut and L. Desvillettes \cite{Bouchut2} that treated the case of divergence free vector-field with continuous coefficient, and to \cite{MH} in which this result is extended to vector-field with $L_{loc}^{2}$ coefficients with a condition of regularity on the direction of the vector-field. We refer the readers to two excellent summaries in 
\cite{ambrisio2} and \cite{lellis}. For some recent developments see \cite{Cara} and  \cite{sies}. 

\medskip

In contrast with its deterministic counterpart, the singular stochastic continuity/transport equation with multiplicative noise is well-posed.  
The addition of a stochastic noise is often used to account for
numerical, empirical or physical uncertainties. The questions of regularizing effects and well-posedness by noise for (stochastic)
partial differential equations have attracted much interest in recent years. In \cite{AttFl11,Beck,Fre1,Fre2,FGP2,MNP14,Moli,NO}, well-posedness and regularization by linear multiplicative noise for continuity/transport equations, that is for

\begin{equation}\label{tr}
  \partial_t u(t, x) + Div  \big( ( b(t, x) + \frac{d B_{t}}{dt}) \cdot  u(t, x)  \big )= 0, 
\end{equation}

have been obtained. We refer to \cite{Moli} for more details on the literature. 

We report here the observation that a 
multiplicative noise as the one used in (\ref{tr})  is not enough to improve 
the regularity of solutions of the following stochastic conservation law 
$$
    \partial_t u(t, x) +  \partial_x u(t, x)  \big( u(t,x) + \frac{d B_{t}}{dt}(\omega)\big ) = 0 \, .
$$
Indeed, for this equation one can observe the appearance of shocks in finite time, just as for the deterministic  conservation law, 
see \cite{Flan1}. For a different approach related to stochastic scalar conservation laws, we address the reader 
to \cite{lions} and \cite{Lions2}. \\

 One of the gaps  the theory
has, in order to work with  the nonlinear systems where $b$ depend on some quantity of $u$, is that  in all  cases $b$
 is  deterministic vector field.

The purpose of the present paper is a
contribution to the following general question: can
one hope for an existence/uniqueness
theory in the case where $b$ is a stochastic process and  it has low regularity? We present one  positive result. More precisely. 
 we  study the continuity equation

\begin{equation}\label{trasport}
 \left \{
\begin{aligned}
    &\partial_t u(t, x) + Div  \big( ( b(t,x, \omega) + \frac{d B_{t}}{dt}) \cdot  u(t, x)  \big )= 0 \, ,
    \\[5pt]
    &u|_{t=0}=  u_{0} \, .
\end{aligned}
\right .
\end{equation}
Here, $(t,x) \in [0,T] \times \R$, $\omega \in \Omega$ is an element of the probability space $(\Omega, \PP, \calF)$, 
$b:\R_+ \times \R\times \R  \to \R$ is a given vector field, $B_{t}$ is a standard Brownian motion. 
The stochastic integration is to be understood in the Stratonovich sense.

\medskip The novelty  of our results is to  show existence and uniqueness of the solutions for one-dimensional stochastic continuity
equation (\ref{trasport})  when the vector field $b$ has random dependence  
 and when  it is bounded  and  integrable without assumptions on the divergence.  

\medskip
Here, of course, the first 
difficulty is in the existence part, since standard approximation schemes
 in general do not provide  existence without assumptions on the divergence in the class of $L^{2}$-solutions. The result is based on the regularization effect of the Brownian perturbation on the flow of the characteristics equation

\begin{equation}\label{ss}
X_{t}= x + \int_{0}^{t}   b(s, X_{s}(x),\omega) \ ds +  B_{t}.
\end{equation}

For the  key estimates on  the spatial weak derivatives of solutions of the SDE (\ref{ss})  we  use stochastic calculus  technique.

Let us describe in few words the strategy of the  uniqueness proof. It is based  on the fact that one primitive $V$ is regular and verifies
the transport equation
\begin{equation}\label{aux}
    \partial_t V(t, x) +  ( b(t,x,\omega) + \frac{d B_{t}}{dt}) \cdot  \nabla V(t, x)  = 0 \, .
\end{equation}

Then using a modified version of the commutator Lemma and the characteristics systems associated to the SPDE (\ref{aux}) we shall show that $V=0$ with initial condition equal to zero, which implies that
$u=0$.

Other pint in this paper is to show  a stability result for the Cauchy problem (\ref{trasport}). 

\medskip

Finally, we point two previous results on regularization by noise  
with random drift. In \cite{Cate} the authors obtain regularization by noise 
for ODEs with  random drift, for general noise including fractional Brownian motion, using 
rough path theory, however, the authors assume that the drift is in suitable Besov space. In \cite{Dub} the authors extend the It\^o-Tanaka trick  for random drift using Malliavin calculus and FBSDEs theory, in that paper the authors assume some
 Malliavin differentiability of the drift. 
\subsection{Hypothesis.}

In this paper we assume the following hypothesis:

\begin{hypothesis}\label{hyp1}
The vector field $b$ satisfies
\begin{equation}\label{cond3-1}
  b\in L^{\infty}(\Omega \times [0,T], L^{1}(\R)),
\end{equation}

\begin{equation}\label{cond3-3}
  b\in L^{\infty}(\Omega \times [0,T] \times  \R),
\end{equation}

\begin{equation}\label{cond3-2}
  b(t,x,\omega)= \int_{0}^{t} f(s,x,\omega) ds  + \int_{0}^{t} g(s,x,\omega) dB_s, 
\end{equation}

\noindent where 

\[
f \in L^{\infty}( \Omega , L^{1}([0,T] \times\R)) 
\]

and 

\[
g\in L^{\infty}( \Omega, L^{1}([0,T],L^{\infty}(\R) ))\cap L^{\infty}( \Omega \times [0,T],L^{1}(\R)).
\]

Moreover, the initial condition verifies 
\begin{equation}\label{weight}
  u_0 \in  L^2(\R)\cap L^1(\R) .
\end{equation}
\end{hypothesis}

\subsection{Notations}First, through of this paper, we fix a stochastic basis with a
$d$-dimensional Brownian motion $\big( \Omega, \mathcal{F}, \{
\mathcal{F}_t: t \in [0,T] \}, \mathbb{P}, (B_{t}) \big)$. Then, we recall to help the intuition, the following definitions 

$$
\begin{aligned}
\text{Itô:}&  \ \int_{0}^{t} X_s dB_s=
\lim_{n    \rightarrow \infty}   \sum_{t_i\in \pi_n, t_i\leq t}  X_{t_i}    (B_{t_{i+1} \wedge t} - B_{t_i}),
\\[5pt]
\text{Stratonovich:}&  \ \int_{0}^{t} X_s  \circ dB_s=
\lim_{n    \rightarrow \infty}   \sum_{t_i\in \pi_n, t_i\leq t} \frac{ (X_{t_i \wedge t   } + X_{t_i} ) }{2} (B_{t_{i+1} \wedge t} - B_{t_i}),
\\[5pt]
\text{Covariation:}& \ [X, Y ]_t =
\lim_{n    \rightarrow \infty}   \sum_{t_i\in \pi_n, t_i\leq t} (X_{t_i \wedge t   } - X_{t_i} )  (Y_{t_{i+1} \wedge t} - Y_{t_i}),
\end{aligned}
$$
where $\pi_n$ is a sequence of finite partitions of $ [0, T ]$ with size $ |\pi_n| \rightarrow 0$ and
elements $0 = t_0 < t_1 < \ldots  $. The limits are in probability, uniformly in time
on compact intervals. Details about these facts can be found in Kunita 
\cite{Ku2}. 

\subsection{It\^o-Wentzell-Kunita formula} 

We consider $X(t,,x, \omega)$ be a continous  $C^{3}$-process and 
continuous  $C^{2}$-semimartingale , for $x\in \R$, and $t\in [0,T]$ of the form 

\[
  X(t,x,\omega)=X_{0}(x) + \int_{0}^{t} f(s,x,\omega) ds  + \int_{0}^{t} g(s,x,\omega) \circ dB_s, 
 \]

If $Y(t,\omega)$   is  a continuous semimartingale, then $X(t,Y_{t},\omega)$
is a continuous semimartingale and the following formula holds

\[
  X(t,Y_{t},\omega)=X_{0}(Y_{0}) + \int_{0}^{t} f(s,Y_{s},\omega) ds  + \int_{0}^{t} g(s,Y_{s},\omega) \circ dB_s
\]
\[
+ \int_{0}^{t} (\nabla_{x}X)(s,Y_{s},\omega)\circ dY_{s},
\]

For more details see  Theorem 8.3 of \cite{Ku2}.

\subsection{Motivations.}
Apart from the theoretical importance of such
an  extension,  the  main  motivation  comes  from  the  study  of  many  nonlinear  partial
differential equations of the mathematical physics.  In various physical models of the
mechanics of fluids it is essential to deal with densities or with velocity fields which
are  not  smooth and this corresponds to effective real world situations. This is our motivation to study the effect of the noise in transport/continuity equation 
when the drift $b$ is a stochastic process.  In particular, we are interested in to show uniqueness of weak solutions for non-local conservation law, that is for the equation

\[
  \partial_t u(t, x) + Div  \big( ( F(t,x,(K\ast u)(x))  + \frac{d B_{t}}{dt}) \cdot  u(t, x)  \big )= 0, 
\]

\noindent where $K$ is a regular kernel. Our conjecture is the uniqueness for type of the conservation law in  opposition to the deterministics theory where the solutions are unique under entropy condition. 

\section{Estimation for the  flow.}

To begin with , let us consider the stochastic differential equation in $\R$, that is to say,
given $s \in [0,T ]$  and  $x \in \mathbb{R}$, we consider 
\begin{equation}\label{itoass}
X_{s,t}(x)= x + \int_{s}^{t}   b(u, X_{s,u}(x),\omega) \ du +  B_{t}-B_{s},
\end{equation}
where $X_{s,t}(x)= X(s,t,x)$ (also $X_{t}(x)= X(0,t,x)$).
In particular, for $m \in \N$ and $0 < \alpha < 1$, we assume
\begin{equation}
\label{DRIFTREGULAR}
   b \in L^1([0,T]; (C^{m,\alpha}(\R))). 
\end{equation}
It is well known that, under the above regularity 
of the drift vector field $b$, the 
stochastic flow $X_{s,t}$ is a $C^m$ diffeomorphism
(see for example  \cite{Chow,Ku}). Moreover, the 
inverse $Y_{s,t}:=X_{s,t}^{-1}$ 
satisfies the following backward stochastic
differential equations,
\begin{equation}\label{itoassBac}
Y_{s,t}= y - \int_{s}^{t}   b(u, Y_{u,t}) \ du  - (B_{t}-B_{s}),
\end{equation}
for $0\leq s\leq t$.

SDEs with discontinuous coefficients and driven by Brownian motion  have been an important area of study in 
stochastic analysis and  has been a very active topic of research in the last years.The method of stochastic characteristics  may be employed to prove uniqueness of solutions of  the stochastic transport/ continuity equation
under weak regularity hypotheses on the drift coefficient. 
 In this way we have the next estimation. 

\begin{lemma}\label{lemma pe}
Assume   $b\in C_b^{\infty}(\R)$ and that satisfies the hypothesis \ref{hyp1}. Then for  $T>0$  there exists a  constant 
$C$  such that
\begin{align}\label{eq0}
\mathbb{E}\bigg[\bigg|\frac{d}{dx}X_{s,t}(x)\bigg|^{-1}\bigg] \leq C.
\end{align}

where C depends on  $T$, $\|b\|_{L^{\infty}([0,T]\times \Omega, L^1(\R))}$, 
$ \|b\|_{  L^{\infty}(\Omega \times [0,T] \times  \R) }$, $ \|f\|_{L^{\infty}(\Omega, L^{1}([0,T]\times \R))}$,
$\|g\|_{L^{\infty}(\Omega,L^{1}([0,T],L^{\infty}(\R)) )}$,  $  \|g\|_{L^{\infty}(\Omega\times [0,T],L^{1}(\R))}$.

\end{lemma}

\begin{proof}
We consider the SDE associated to the vector field $b$ :
\begin{equation*}
d X_t = b (t,X_t,\omega) \, dt + d B_t \, ,  \hspace{1cm}   X_s = x \,.
\end{equation*}

We note that $\partial_x X_{s,t}$ satisfies

\[
\partial_x X_{s,t}=\exp\bigg\{ \int_{s}^{t}  b^{\prime} (X_{s,u}) \  du  \bigg\}. 
\]

We  denoted

 $$\tilde{b}(t,z,\omega)=\int_{\infty}^z b(t,y,\omega)dy,$$

 $$\tilde{g}(t,z,\omega)=\int_{\infty}^z g(t,y,\omega)dy,$$

$$\tilde{f}(t,z,\omega)=\int_{\infty}^z f(t,y,\omega)dy.$$

 Applying the It\^o-Wentzell-Kunita formula   to $\tilde{b}$ , see Theorem 8.3 of \cite{Ku2}, we have

\[
\tilde{b}(t,X_{s,t},\omega)=\tilde{b}(s,x,\omega)+ \int_s^t \tilde{f}(u,X_{s,u},\omega) \ du +
 \int_s^t \tilde{g}(u,X_{s,u},\omega) \ dB_u +\int_s^t  b^{2}(u,X_{s,u},\omega) \ du  
\]
\begin{align*}
+ \int_s^t  g(u,X_{s,u},\omega) \ du + \int_s^t  b(u,X_{s,u},\omega) \ dB_u +  \frac{1}{2} \int_s^t b'(u,X_{s,u},\omega)  \ du\,.
\end{align*}

\noindent We set
\begin{align*}
\mathcal{E}\bigg(\int_s^t b(u,X_{s,u},\omega)dB_u\bigg)=\exp\bigg\{\int_s^t b(u,X_{s,u},\omega)dB_u-\frac{1}{2} \int_s^{t} b^2(u,X_{s,u},\omega)du \bigg\}, 
\end{align*}

and

\begin{align*}
\mathcal{E}\bigg(\int_s^t \tilde{g}(u,X_{s,u},\omega)dB_u\bigg)=\exp\bigg\{\int_s^t \tilde{g}(u,X_{s,u},\omega)dB_u-\frac{1}{2} \int_s^{t} \tilde{g}^{2}(u,X_{s,u},\omega)du \bigg\}. 
\end{align*}

 Now, we observe

\begin{equation}\label{uno}
\|\tilde{b}\|_{\infty}\leq \|b\|_{L^{\infty}([0,T]\times \Omega, L^1(\R))} ,
\end{equation}

\begin{equation}\label{dos}
\| \int_s^t  b^{2}(X_{s,u}) \ du \|_{\infty}\leq C  \|b\|_{  L^{\infty}(\Omega \times [0,T] \times  \R) }^{2} ,
\end{equation}

\begin{equation}\label{tres}
\| \int_s^t  \tilde{f}(u,X_{s,u},\omega) \ du \|_{\infty}\leq C  \|f\|_{L^{\infty}(\Omega, L^{1}([0,T]\times \R))},
\end{equation}

\begin{equation}\label{cuat}
\| \int_s^t  g(u,X_{s,u},\omega) \ du \|_{\infty}\leq C \|g\|_{L^{\infty}(\Omega,L^{1}([0,T],L^{\infty}(\R)) )},
\end{equation}

\begin{equation}\label{cin}
\| \int_s^t  |\tilde{g}(u,X_{s,u},\omega)|^{2} \ du \|_{\infty}\leq C  \|g\|_{L^{\infty}(\Omega\times [0,T],L^{1}(\R))}^{2}.
\end{equation}

 From (\ref{uno}), (\ref{dos}), (\ref{tres}), (\ref{cuat}), (\ref{cin}) and   H\"older inequality
we have

\begin{align}
\E\bigg[\bigg|\frac{d X_{s,t}}{dx}(x)\bigg|^{-1}\bigg ] & =
\E\bigg[\exp\bigg\{ \frac{1}{2}\bigg[- \tilde{b}(t,X_{s,t},\omega) +\tilde{b}(s,x,\omega)  +\int_s^t  b^{2}(u,X_{s,u},\omega) \ du   \nonumber\\
&  + \int_s^t b(u,X_{s,u},\omega) dB_u +  \int_s^t \tilde{f}(u,X_{s,u},\omega) \ du + \nonumber\\& 
+ \int_s^t g(u,X_{s,u},\omega) \ du + \int_s^t \tilde{g}(u,X_{s,u},\omega) dB_u \bigg]\bigg\} 
\nonumber\\&  \leq C
\E\bigg[\exp\bigg\{ \frac{1}{2}\bigg[\int_s^t b(u,X_{s,u},\omega) dB_u- \frac{1}{2}\int_s^t  b^{2}(u,X_{s,u},\omega) \ du   \nonumber\\
&  +  \int_s^t \tilde{g}(u,X_{s,u},\omega) dB_u- \frac{1}{2} \int_s^t \tilde{g}^{2}(u,X_{s,u},\omega) \ du  \bigg]\bigg\} 
\nonumber\\& 
    \leq C \E \bigg[ \mathcal{E}\bigg(\int_s^t b(u,X_{s,u},\omega)dB_u\bigg) \bigg] \times
 \E \bigg[ \mathcal{E}\bigg(\int_s^t \tilde{g}(u,X_{s,u},\omega)dB_u\bigg) \bigg]
\nonumber\\ & 
\end{align}

Finally we observe  that the  processes  $ \mathcal{E}\bigg(\int_0^t b(u,X_{s,u},\omega)dB_s\bigg)$,
 $ \mathcal{E}\bigg(\int_0^t \tilde{g}(u,X_{s,u},\omega)dB_s\bigg)$  are martingales with expectation equal to one.  From this  we conclude our lemma.
\end{proof}

\begin{remark}The same results is valid for the backward flow $Y_{s,t}$  since it is  solution of the same SDE driven by 
the drifts $-b$.
\end{remark}

\section{$L^{2}$- Solutions.}

\subsection{Definition of solutions}
\begin{definition}\label{defisoluH}
A stochastic process $u\in   L^\infty([0,T],L^{2}( \Omega\times \R)) \cap L^{1}(\Omega\times [0,T]\times \R)$  is called a  $L^{2}$- weak solution of the Cauchy problem \eqref{trasport} when: For any $\varphi \in C_0^{\infty}(\R)$, the real valued process $\int  u(t,x)\varphi(x)  dx$ has a continuous modification which is an $\mathcal{F}_{t}$-semimartingale, and for all $t \in [0,T]$, we have $\mathbb{P}$-almost surely
\begin{equation} \label{DISTINTSTR}
\begin{aligned}
    \int_{\R} u(t,x) \varphi(x) dx = &\int_{\R} u_{0}(x) \varphi(x) \ dx
	+ \int_{0}^{t} \!\! \int_{\R}   u(s,x)   \, b(s,x,\omega) \partial_x \varphi(x)  dx ds
\\[5pt]
    & + \int_{0}^{t} \!\! \int_{\R}   u(s,x) \ \partial_x \varphi(x) \ dx \, {\circ}{dB_s} \, .
\end{aligned}
\end{equation}
\end{definition}

\begin{remark}\label{lemmaito}
Using the same idea as in Lemma 13 \cite{FGP2}, one can write the problem (\ref{trasport}) in It\^o form as follows, a  stochastic process $u\in   L^\infty([0,T],L^{2}( \Omega\times \R))\cap L^{1}(\Omega\times [0,T]\times \R) $ is  a $ L^{2}$- weak solution  of the SPDE (\ref{trasport}) iff for every test function $\varphi \in C_{0}^{\infty}(\mathbb{R})$, the process $\int u(t, x)\varphi(x) dx$ has a continuous modification which is a $\mathcal{F}_{t}$-semimartingale and satisfies the following It\^o's formulation

\[
\begin{aligned}
    \int_{\R} u(t,x) \varphi(x) dx = &\int_{\R} u_{0}(x) \varphi(x) \ dx
	+ \int_{0}^{t} \!\! \int_{\R}   u(s,x)   \, b(s,x,\omega) \partial_x \varphi(x) \ dx ds
\\[5pt]
    & + \int_{0}^{t} \!\! \int_{\R}   u(s,x) \ \partial_x \varphi(x) \ dx \, dB_s \,  + \frac{1}{2} \int_{0}^{t} \!\! \int_{\R}   u(s,x) \ \partial_x^{2} \varphi(x) \ dx \, ds.
\end{aligned}
\]

\end{remark}


\subsection{Existence.}

The goal of this section is to prove general existence result for stochastic continuity equation without assumptions on the divergence.

\begin{lemma}
Assume that hypothesis \ref{hyp1} holds. Then there exist $L^2$-weak solutions of the Cauchy problem \eqref{trasport}.
\end{lemma}

\begin{proof}

{\it Step 1: Regularization.}

 Let $\{\rho_\varepsilon\}_\varepsilon$ be a family of standard symmetric mollifiers and $\eta$ a nonnegative smooth cut-off function supported on the ball of radius 2 and such that $\eta=1$ on the ball of radius 1. Now, for every $\varepsilon>0$, we introduce the rescaled functions $\eta_\varepsilon (\cdot) = \eta(\varepsilon \cdot)$. Thus, we  define the family of  regularized coefficients given by
$$b^{\epsilon}(x) = \eta_\varepsilon(x) ( b \ast \rho_\varepsilon  (x)) $$
and
$$u_0^\varepsilon (x) = \eta_\varepsilon(x) \big(  u_0 \ast \rho_\varepsilon (x)  \big) \,.$$
Clearly we observe that, for every  $\varepsilon>0$, any element $b^{\varepsilon}$, $u_0^\varepsilon$ are smooth (in space) and have  compactly supported with bounded derivatives of all orders.  We consider the  regularized version of  stochastic continuity  equation given by :
\begin{equation}\label{STE-reg}
 \left \{
\begin{aligned}
    &d u^\varepsilon (t, x) +  Div\bigg(  u^\varepsilon (t, x)  \cdot \big( b^\varepsilon (t,x,\omega)  dt +
 \circ d B_{t} \big) \bigg)= 0\, ,
    \\[5pt]
    &u^\varepsilon \big|_{t=0}=  u_{0}^\varepsilon.
\end{aligned}
\right .
\end{equation}
Following the classical theory of H. Kunita \cite[Theorem 6.1.9]{Ku} we obtain that

\[
u^{\varepsilon}(t,x) =  u_{0}^{\varepsilon} (\psi_t^{\varepsilon}(x))  J\psi_t^{\varepsilon}(x)
\]
is the unique solution to the regularized equation \eqref{STE-reg}, where $\phi_t^{\varepsilon}$ is the flow associated  to the following stochastic differential equation (SDE):
\begin{equation*}
d X_t = b^\varepsilon (X_t) \, dt + d B_t \, ,  \hspace{1cm}   X_0 = x \,,
\end{equation*}
and $\psi_t^{\varepsilon}$ is the inverse of $\phi_t^{\varepsilon}$.

\bigskip

{\it Step 2: Boundedness.} Making the change of variables $y=\psi_t^{\varepsilon}(x)=(\phi_t^{\varepsilon}(x))^{-1}$ we have that
\begin{align*}
\int_{\R} \E[|u^{\varepsilon}(t,x)|^2]\,  dx & =\E \int_{\R} |u_{0}^{\varepsilon} (y)|^2  (J\phi_t^{\varepsilon}(x))^{-1}  dx.
\end{align*}

Now, by Lemma \ref{lemma pe} we obtain 
\begin{align}\label{eq0}
\int_{\R} \E[|u^{\varepsilon}(t,x)|^2]\,  dx \leq C \int_{\R} |u_{0}^{\varepsilon} (y)|^2  dx.
\end{align}

Therefore, the sequence $\{u^{\varepsilon}\}_{\varepsilon>0}$ is bounded in $u\in L^2(\Omega\times [0,T]\times \R )\cap L^\infty([0,T],L^{2}( \Omega\times \R))$  . Then  there exists a convergent subsequence, which we denote also by $u^{\varepsilon}$, such that converge weakly in $L^2(\Omega\times [0,T]\times \R)$ and weak-star in $L^\infty([0,T],L^{2}( \Omega\times \R))$ to some process $u\in L^2(\Omega\times [0,T]\times \R)\cap L^\infty([0,T],L^{2}( \Omega\times \R))$. Since this subsequence is bounded in 
$L^{1}(\Omega\times [0,T]\times \R)$ we follows that $u^{\varepsilon}$ converge to the measure $\mu$ and $\mu=u. $

\bigskip
{\it Step 3: Passing to the Limit.}
Now, if $u^{\varepsilon}$ is a solution of \eqref{STE-reg}, it is also a weak solution, that is, for any test function $\varphi\in C_0^{\infty}(\R)$, $u^{\varepsilon}$ verifies  (written in the Itô form):
\begin{align*}
\int_{\R} u^{\varepsilon}(t,x) \varphi(x) dx = &\int_{\R} u^{\varepsilon}_{0}(x) \varphi(x) \ dx + \int_{0}^{t} \!\! \int_{\R}   u^{\varepsilon}(s,x)   \, b^{\varepsilon}(s,x,\omega) \partial_x \varphi(x) \ dx ds \\
    & + \int_{0}^{t} \!\! \int_{\R}   u^{\varepsilon}(s,x) \ \partial_x \varphi(x) \ dx \, dB_s \,  + \frac{1}{2} \int_{0}^{t} \!\! \int_{\R}   u^{\varepsilon}(s,x) \ \partial_x^{2} \varphi(x) \ dx \, ds\,.
\end{align*}
Then, for prove existence of the equation  \eqref{trasport} is enough to pass to the limit in the above equation along the convergent subsequence found. This is made through of the same arguments of \cite[theorem 15]{FGP2}.

\end{proof}

\subsection{Uniqueness.}

In this section, we shall present a uniqueness theorem
for the SPDE (\ref{trasport}).  The proof is based on the characteristic method and
the commutator Lemma  to a primitive of the solution. We pointed that similar  arguments was used in previous work  
\cite{Moli}.

\begin{theorem}\label{uni2}
Under the conditions of hypothesis \ref{hyp1}, uniqueness holds for  $L^{2}$- weak solutions of the Cauchy problem \eqref{trasport} in the following sense: if $u,v$ are $L^{2}$- weak solutions with the same initial data $u_{0}\in  L^2(\R)\cap L^{1}(\R)$, then  $u= v$ almost everywhere in $ \Omega  \times [0,T] \times \R$.
\end{theorem}

\begin{proof}

{\it Step 0: Set of solutions.} We remark that the set of  $L^{2}$- weak solutions is a linear subspace of $L^{2}( \Omega\times[0, T]\times \mathbb{R})$, because the stochastic continuity equation is linear, and the regularity conditions is a linear constraint. Therefore, it is enough to show that a $L^{2}$- weak solution $u$ with initial condition $u_0= 0$ vanishes identically.

\bigskip

{\it Step 1:  Primitive of the solution.} We set

\[
V(t,x)=\int_{-\infty}^{x} u(t,y) \ dy.
\]

 We observe that $\partial_x V(t,x)=u(t,x)$  belong to $L^{2}( \Omega\times[0, T]\times \mathbb{R} )$. Now, we consider a  nonnegative smooth cut-off function $\eta$ supported on the ball of radius 2 and such that  $\eta=1$ on the ball of radius 1. For any $R>0$, we introduce the rescaled functions $\eta_R (\cdot) =  \eta(\frac{.}{R})$.

\bigskip

 For all  test functions  $\varphi\in C_0^{\infty}(\R)$ we obtain 
\[
 \int_{\R} V(t,x) \varphi(x) \eta_R (x)  dx = - \int_{\R} u(t,x)   \theta(x) \eta_R (x)  dx
-\int_{\R} V(t,x)   \theta(x) \partial_x \eta_R (x)  dx\,,
\]

\noindent where $\theta(x) =\int_{-\infty}^{x}  \varphi(y)   \ dy$. By  definition  of $L^{2}$-solutions , taking as test function $ \theta(x) \eta_R (x)$ we get 

\begin{align}\label{DISTINTSTRTR}
    \int_{\R} & V(t,x) \ \eta_R (x) \varphi(x) dx = - \int_{0}^{t} \!\! \int_{\R}   \partial_x V(s,x)   \, b (s,x,\omega) \eta_R (x) \varphi(x) \ dx ds \nonumber\\[5pt]
    & - \int_{0}^{t} \!\! \int_{\R}   \partial_x V(s,x) \  \eta_R (x) \varphi(x) \ dx \, {\circ}{dB_s} - \int_{0}^{t} \!\! \int_{\R}   \partial_x V(s,x)   \, b(s,x,\omega) \partial_x \eta_R (x) \theta(x)  \ dx ds\nonumber\\[5pt]
    &- \int_{0}^{t} \!\! \int_{\R}   \partial_x V(s,x) \  \partial_x \eta_R (x) \theta(x)  \ dx \, {\circ}{dB_s}-\int_{\R} V(t,x)   \theta(x) \partial_x \eta_R (x)  dx.
\end{align}

We observe that

\[
\int_{0}^{t} \!\! \int_{\R}   \partial_x V(s,x) \  \partial_x \eta_R (x) \theta(x)  \ dx \, {\circ}{dB_s}\rightarrow 0,
\]

\[
\int_{\R} V(t,x)   \theta(x) \partial_x \eta_R (x)  dx \rightarrow 0,
\]

as $R\rightarrow \infty$. Taking $L^{2}([0,T]\times \Omega)$-limit  in equation  (\ref{DISTINTSTRTR})  we have that

		\[
		\int_{\R} V(t,x) \varphi(x) dx 
		\]
		
\[
	=- \int_{0}^{t} \!\! \int_{\R}   \partial_x V(s,x)   \,  b(s,x,\omega)  \varphi(x) \ dx ds
- \int_{0}^{t} \!\! \int_{\R}   \partial_x V(s,x) \   \varphi(x) \ dx \, {\circ}{dB_s} .
\]
\bigskip

{\it Step 2: Smoothing.}
Let $\{\rho_{\varepsilon}(x)\}_\varepsilon$ be a family of standard symmetric mollifiers. For any $\varepsilon>0$ and $x\in\R^d$ we use $\rho_\varepsilon(x-\cdot)$ as test function, then we deduce 
$$
\begin{aligned}
      \int_{\R} V(t,y) \rho_\varepsilon(x-y) \, dy  = &\, - \int_{0}^{t}  \int_{\R} \big( b(s,y,\omega)  \partial_y V(s,y)  \big)  \rho_\varepsilon(x-y) \ dy ds
		\\[5pt]
    & -  \int_{0}^{t} \!\! \int_{\R} \partial_y V(s,y) \, \rho_\varepsilon(x-y)  \, dy \circ dB_s
\end{aligned}
$$

We denote   $V_\varepsilon(t,x)= (V\ast \rho_\varepsilon)(x)$, $b_\varepsilon(t,x,\omega)= (b \ast \rho_\varepsilon)(x)$ and
$(bV)_\varepsilon(t,x)= (b.V\ast \rho_\varepsilon)(x)$. Thus we have

\[
    V_{\varepsilon}(t,x) + \int_{0}^{t} b_{\epsilon}(s,x,\omega)  \partial_x V_{\varepsilon}(s,x) \,  ds   +  \int_{0}^{t}   \partial_{x}  V_{\varepsilon}(s,x) \, \circ dB_s
\]		
	
\[         
=\int_{0}^{t} \big(\mathcal{R}_{\epsilon}(V,b) \big) (x,s) \,  ds ,
\]

\noindent where we denote
$ \mathcal{R}_{\epsilon}(V,b)  = b_\varepsilon \ \partial_x V_\varepsilon  -  (b\partial_x V)_\varepsilon  $.

\bigskip

{\it Step 3: Method of Characteristics.}
Applying the It\^o-Wentzell-Kunita formula   to $ V_{\varepsilon}(t,X_{t}^{\epsilon})$
, see Theorem 8.3 of \cite{Ku2}, we have

\[
   V_{\varepsilon}(t,X_{t}^{\epsilon})  = \int_{0}^{t} \big(\mathcal{R}_{\epsilon}(V,b) \big) (X_s^{\epsilon},s)  ds  .
\]

Then, considering that $X_{t}^{\epsilon}=X_{0,t}^{\epsilon}$ and $Y_{t}^{\epsilon}=Y_{0,t}^{\epsilon}=(X_{0,t}^{\epsilon})^{-1}$ 
we deduce that 

\[
   V_{\varepsilon}(t,x)  = \int_{0}^{t} \big(\mathcal{R}_{\epsilon}(V,b) \big) (X_{0,s}^{\epsilon}(Y_{0,t}^{\epsilon}),s)  ds=\int_{0}^{t} \big(\mathcal{R}_{\epsilon}(V,b) \big) (Y_{s,t}^{\epsilon},s)  ds  .
\]

Multiplying by the test functions $\varphi$ and integrating in $\R$ we get

\begin{equation}
   \int V_{\varepsilon}(t,x) \ \varphi(x) dx  =
	\int_{0}^{t}  \int \big(\mathcal{R}_{\epsilon}(V,b) \big) (Y_{s,t}^{\epsilon},s) \  \  \varphi(x) \ \, dx \   ds  .
\end{equation}

 \noindent Finally  we observe that

\begin{equation}
	\int_{0}^{t}  \int \big(\mathcal{R}_{\epsilon}(V,b) \big) (Y_{s,t}^{\epsilon},s) \ \varphi(x) \ \, dx \   ds =
	\int_{0}^{t}  \int \big(\mathcal{R}_{\epsilon}(V,b) \big) (x,s) \   JX_{s,t}^{\epsilon}  \varphi(X_{s,t}^{\epsilon}) \ \, dx \   ds .
\end{equation}

\bigskip

{\it Step 4: Convergence of the commutator.}   Now, we observe that $\mathcal{R}_{\epsilon}(V,b)$ converge to zero in
$L^{2}([0,T]\times \Omega\times \R )$. In fact, we  get that

\[
  (b \ \partial_x V)_{\varepsilon} \rightarrow b \ \partial_x V \ in \ L^{2}([0,T] \times  \Omega \times \R ).
\]

\noindent Moreover,  we have 

\[
b_{\epsilon} \rightarrow b \  a.e,
\]

\[
b \ is \ bounded
\]

and

\[
\partial_x V_{\epsilon} \rightarrow   \partial_x V \ in \ L^{2}([0,T]\times \Omega\times \R ).
\]

Then by the dominated convergence theorem we obtain

\[
b_{\epsilon}   \partial_x V_{\varepsilon} \rightarrow b \ \partial_x V \ in \ L^{2}([0,T]\times  \Omega \times \R ).
\]

\bigskip

{\it Step 5: Conclusion.}  From step 3  we have 

\begin{equation}\label{conv}
   \int V_{\varepsilon}(t,x) \ \varphi(x) dx  =
		\int_{0}^{t}  \int \big(\mathcal{R}_{\epsilon}(V,b) \big) (x,s) \   JX_{s,t}^{\epsilon}  \varphi(X_{s,t}^{\epsilon}) \ \, dx \   ds .
\end{equation}

 By   H\"older's inequality we obtain 

\[
	\E \bigg|\int_{0}^{t}  \int \bigg(\mathcal{R}_{\epsilon}(V,b) \bigg) (x,s) \   JX_{s,t}^{\epsilon}  \varphi(X_{s,t}^{\epsilon}) \ \, dx \   ds \bigg|
	\]
	
	\[
	\leq   	 \bigg(\E \int_{0}^{t}  \int |\big(\mathcal{R}_{\epsilon}(V,b) \big) (x,s)|^{2} \  \, dx \   ds \bigg)^{\frac{1}{2}}
  \bigg(\E \int_{0}^{t}  \int  | JX_{s,t}^{\epsilon} \varphi(X_{s,t}^{\epsilon})|^{2} \ \, dx \ ds \bigg)^{\frac{1}{2}}
\]

\noindent  From step 4 we follow

\[
 \bigg(\E \int_{0}^{t}  \int |\big(\mathcal{R}_{\epsilon}(V,b) \big) (x,s)|^{2} \  \, dx \   ds \bigg)^{\frac{1}{2}}\rightarrow 0.
\]

From Lemma \ref{lemma pe} and Remark 2.2 we deduce 

\[
\bigg(\E \int_{0}^{t}  \int  | JX_{s,t}^{\epsilon} \varphi(X_{s,t}^{\epsilon})|^{2} \ \, dx \ ds \bigg)^{\frac{1}{2}}
=\bigg(\E \int_{0}^{t}  \int  | JY_{s,t}^{\epsilon}|^{-1} |\varphi(x)|^{2} \ \, dx \ ds \bigg)^{\frac{1}{2}}
\]
\[
\leq C  \bigg( \int |\varphi(x)|^{2} \ \, dx \  \bigg)^{\frac{1}{2}},
\]

Passing to the limit in equation (\ref{conv})   we deduced that  $V=0$. Then we conclude  that  $u=0$.

\end{proof}

\subsection{Stability}

To end up the well-posedness for the continuity equation
\eqref{trasport}, it remains to show the stability 
 property for the solution with respect to the initial datum.  We 
use the same ideas  that in the uniqueness proof.

\begin{theorem}\label{estaweak} Assume hypothesis \ref{hyp1}.  
Let $\{u_0^n\}$ be any sequence, with $u_0^n \in L^2(\R)\cap L^1(\R)$ $(n \geq 1)$, converging strongly  to 
$u_0 \in L^2(\R)\cap L^1(\R)$. Let $u(t,x)$, $u^n(t,x)$ be the unique weak $L^{2}-$solution of the Cauchy problem \eqref{trasport},
for respectively the initial data $u_0$ and $u_0^{n}$. Then, for all 
 $t \in  [0, T ]$, and for each function $\varphi \in C_c^\infty(\R^d)$ $\mathbb{P}-$ a.s.
$$
   \int_{\R^d} u^n(t,x) \, \vp(x) \ dx \quad \text{converges to} \quad  \int_{\R^d} u(t,x) \, \vp(x) \ dx \quad \text{$\mathbb{P}-$ a.s.}.
$$
\end{theorem}

\begin{proof}
   We set

\[
V(t,x)=\int_{-\infty}^{x} u(t,y) \ dy,
\]

and

\[
V^{n}(t,x)=\int_{-\infty}^{x} u^{n}(t,y) \ dy.
\]
Then we have

\[
    \int_{\R} V(t,x) \varphi(x) dx 
	\]
		
\[		
	= \int u_{0}(x) \varphi(x) dx - \int_{0}^{t} \!\! \int_{\R}   \partial_x V(s,x)   \,  b(s,x,\omega)  \varphi(x) \ dx ds
- \int_{0}^{t} \!\! \int_{\R}   \partial_x V(s,x) \   \varphi(x) \ dx \, {\circ}{dB_s},
\]

and 

\[
    \int_{\R} V^{n}(t,x) \varphi(x) dx
		\]
		
		\[
	=\int u_{0}^{n}(x) \varphi(x) dx- \int_{0}^{t} \!\! \int_{\R}   \partial_x V^{n}(s,x)   \,  b(s,x,\omega)  \varphi(x) \ dx ds
- \int_{0}^{t} \!\! \int_{\R}   \partial_x V^{n}(s,x) \   \varphi(x) \ dx \, {\circ}{dB_s} .
\]

We set  $W^{n}=V(t,x)-V^{n}(t,x)$ and $W_{0}^{n}=u_{0}(x)-u_{0}^{n}(x)$. Thus we obtain  that $W^{n}$ verifies

\[
    \int_{\R} W^{n}(t,x) \varphi(x) dx
		\]
		
		\[
=	\int W_{0}^{n}(x) \varphi(x) dx- \int_{0}^{t} \!\! \int_{\R}   \partial_x W^{n}(s,x)   \,  b(s,x,\omega)  \varphi(x) \ dx ds
- \int_{0}^{t} \!\! \int_{\R}   \partial_x W^{n}(s,x) \   \varphi(x) \ dx \, {\circ}{dB_s} .
\]

We denote   $W_\varepsilon^{n}(t,x)= (W^{n}\ast \rho_\varepsilon)(x)$, $b_\varepsilon(t,x,\omega)= (b \ast \rho_\varepsilon)(x)$ and
$(bW^{n})_\varepsilon(t,x)= (b.V\ast \rho_\varepsilon)(x)$. Thus we have

\[
    W_{\varepsilon}^{n}(t,x) - W_{0}^{n,\epsilon}(x)  + \int_{0}^{t} b_{\epsilon}(s,x,\omega)  \partial_x W_{\varepsilon}^{n}(s,x) \,  ds   + 
		\int_{0}^{t}   \partial_{x}  W_{\varepsilon}^{n}(s,x) \, \circ dB_s =	
\]

  \[       
	\int_{0}^{t} \big(\mathcal{R}_{\epsilon}(W^{n},b) \big) (x,s) \,  ds ,
\]
\noindent where we denote
$ \mathcal{R}_{\epsilon}(W^{n},b)  = b_\varepsilon \ \partial_x W_\varepsilon^{n}  -  (b\partial_x W^{n})_\varepsilon  $.

Arguing as in the uniqueness proof we obtain

\begin{equation}\label{esta}
   \int W_{\varepsilon}^{n}(t,x) \ \varphi(x) dx  = \int W_{0}^{n,\epsilon}(Y_t^{\epsilon})(x) dx  +
	\int_{0}^{t}  \int \big(\mathcal{R}_{\epsilon}(W^{n},b) \big) (Y_{s,t}^{\epsilon},s) \  \  \varphi(x) \ \, dx \   ds  .
\end{equation}

Then  passing to the limit as $\epsilon\rightarrow 0$ and $n\rightarrow\infty$ in equation 
(\ref{esta}) we obtain

\[
  \lim_{n\rightarrow \infty} \int W(t,x) \ \varphi(x) dx  = 0,
\]

and this implies that

\[
  \lim_{n\rightarrow \infty} \int u^{n}(t,x) \ \varphi(x) dx  = \int u(t,x) \ \varphi(x) dx ,
\]

\end{proof}

\subsection{ Negative example} 

We considered the stochastic equation

\[
  \partial_t u(t, x) + Div  \big( ( b( x-B_{t}) + \frac{d B_{t}}{dt}) \cdot  u(t, x)  \big )= 0, 
\]

which is equivalent to the deterministic continuity  equation

\[
  \partial_t u(t, x) + Div  \big( ( b(x) u(t,x))= 0.  
\]

Then we do  not expect to obtain the regularization efect by noise. 

Now, we write the drift term in semimartingale form, appliyng the Ito formula to $b(x-B_t)$ we have

\[  
b(x-B_t)= b(x) - \int_{0}^{t} b^{\prime}(x-B_{t}) dB_{t}+ \frac{1}{2}   \int_{0}^{t} b^{\prime\prime}(x-B_{t}) ds.
\]

\noindent Using the notation in our hypothesis  we have $b^{\prime}=g$ and $b^{\prime\prime}=f$.Thus  we conclude that 
$b$ satisfies our hypothesis only  when it is regular.

\subsection{A possible extension.} 
The main our  tool in order to have estimations on the derivative of the flow was the It\^o-Wentzell-Kunita formula. 	
However,  it is possible  only to apply this  formula for compositions of semimartingales. In order to generalized our result for more general $b$ we have in mind to work in the context of the theory of  stochastic calculus  via regularization. This calculus  was introduced by 
  by F. Russo and P. Vallois ( see \cite{RVSem} as general reference ) and it have been studied and developed by
many authors.  In the paper of F.Flandoli and F. Russo they obtain a  It\^o-Wentzell-Kunita formula
for more general process, see \cite{Flan} for details.

\section*{Acknowledgements}
    Christian Olivera  is partially supported by FAPESP 
		by the grants 2017/17670-0. and 2015/07278-0.


\end{document}